\documentclass[12pt]{amsart}

\usepackage{amsmath}
\usepackage{amssymb}
\usepackage{amsthm}

\theoremstyle{plain} 
\newtheorem{theorem}{Theorem}[section]
\newtheorem{proposition}[theorem]{Proposition}
\newtheorem{lemma}[theorem]{Lemma}
\newtheorem{corollary}[theorem]{Corollary} 

\theoremstyle{remark}

\theoremstyle{definition}
\newtheorem{definition}[theorem]{Definition}

\theoremstyle{remark}
\newtheorem{remark}[theorem]{Remark}

\DeclareMathOperator{\Stab}{Stab}

\DeclareMathOperator{\Orb}{Orb}
\DeclareMathOperator{\ch}{char}
\DeclareMathOperator{\Gal}{Gal}
\DeclareMathOperator{\tr}{tr}

\DeclareMathOperator{\Per}{Per}
\DeclareMathOperator{\Frob}{Frob}

\DeclareMathOperator{\FPP}{FPP}

\newcommand{\fp}{ {\mathfrak p} }

\newcommand{\fq}{ {\mathfrak q} }

\newcommand{\fP}{ {\mathfrak P} }

\newcommand{\cO}{ {\mathcal O} }

\newcommand{\bZ} { {\mathbb Z}}

\newcommand{\F} { {\mathbb F}} 
\newcommand{\bP} { {\mathbb P}} 
\newcommand{\bQ} { {\mathbb Q}}

\begin{document}

\author{Jamie Juul}

\title{The Image Size of Iterated Rational Maps over Finite Fields}

\maketitle

\begin{abstract}
Let $\varphi:\mathbb{P}^1(\mathbb F_q)\to\mathbb{P}^1(\mathbb F_q)$ be a rational map of degree $d>1$ on a fixed finite field.  We give asymptotic formulas for the size of image sets $\varphi^n(\mathbb{P}^1(\mathbb F_q))$ as a function of $n$. This is done using properties of Galois groups of iterated maps, whose connection to the size of image sets is established via the Chebotarev Density Theorem. We apply our results in the following setting. For a rational map defined over a number field, consider the reduction of the map modulo each prime of the number field. We use our results to give explicit bounds on the proportion of periodic points in the residue fields.
\end{abstract}

\section{Introduction}\label{Introduction}

Let $\varphi:\bP^1(\F_q)\to\bP^1(\F_q)$ be a rational map of degree $d>1$ on a fixed finite field and let $\varphi^n$ denote the $n$-th iterate of $\varphi$. We would like to consider the size of the \textbf{image sets} $\varphi^n\left(\bP^1(\F_q)\right) = \{\varphi^n(a) : a\in \bP^1(\F_q)\}$ as $n$ varies and the set of \textbf{periodic points} $\Per(\varphi) = \{a\in \bP^1(\F_q) : \varphi^k(a) = a \text{ for some } k> 0\}$. Image sets are also called \textbf{value sets} and are denoted $V_{\varphi^n}$ by some authors.  
The size of $\varphi^n\left(\bP^1(\F_q)\right)$ is eventually constant, as after a certain number of iterates only the periodic points remain in the image. In this paper, we address the question  of how fast this contraction occurs.

Many authors have investigated the question of the size of the image or value sets for polynomials $f(x)\in \F_q[x]$ which are not necessarily iterates, defined simply as $f(\F_q)=\{f(a):a\in\F_q\}$. Note, when the map is defined by a polynomial, one may simply work with $\F_q$ rather than $\bP^1(\F_q)$ since the point at infinity is fixed. Birch and Swinnerton-Dyer \cite{BS} proved for a degree $d>1$ polynomial $f(x)\in \F_q[x]$, if the Galois group of the splitting field of $f(x)-t$ over $\bar{\F}_q(t)$ is the full symmetric group $S_d$, then
 \begin{equation}\label{eqnBSD}
 \#f(\F_q)=\left(\sum_{k=1}^d\frac{(-1)^{k-1}}{k!}\right)q+O_d({q}^{1/2}),
 \end{equation}
  answering a question of Chowla \cite{chowla}.  Other results in this area have been proven in \cite{GCM, GW, Shao}.
  
  \begin{remark}
  	The big-O notation here is used to mean $\left|\#f(\F_q)-\left(\sum_{k=1}^d\frac{(-1)^{k-1}}{k!}\right)q\right|< Mq^{1/2}$, where $M$ is some constant depending on only $d$, and in particular is independent of $q$ and $f$. Further, this bound holds for all values of $q$.
  	
  	Throughout this paper big-O notation will be used in the same way.  We will use subscripts in big-O notation to denote dependence of the implied constant on the variables in the subscript. In particular, when big-O is used without a subscript, the implied constant is a fixed number (independent of $d, q, n$ and $\varphi$). The estimates hold for all values of $d$, $q$, and $n$, unless otherwise noted.
  \end{remark}

The connection of these problems to Galois theory is established via the Chebotarev density theorem. Specifically, if $t$ is transcendental over $\F_q$, $\#f(\F_q)=Cq + O_d(\sqrt{q})$, where $C$ is the proportion of elements in the Galois group of the splitting field of $f(x)-t$ over ${\F}_q(t)$ fixing some root of $f(x)-t$, provided the extension is \textbf{geometric}, that is, the splitting field of $f^n(x)-t$ does not contain a nontrivial algebraic extension of $\F_q$. The coefficient in Birch and Swinnerton-Dyer's result is precisely the proportion of elements of $S_d$ with a fixed point.  An analogous result holds for a generating coset of the Galois group of $f(x)-t$ over $\bar{\F}_q(t)$ in the Galois group of $f(x)-t$ over ${\F}_q(t)$ for the non-geometric case. This idea was also used by Cohen on work toward this and several related questions \cite{Cohen2}. Odoni used similar methods to study iterated polynomials, although he was looking at a different application \cite{odoni}. We build on some of Odoni's work here.

We fix the following notation. Let $\varphi(x)\in \F_q(x)$ be a rational function with degree $d>1$. Let $K_n=\F_q(\varphi^{-n}(t))$ and $K^*_n=\bar{\F}_q(\varphi^{-n}(t))=K_n\bar{\F}_q$, the splitting fields of $\varphi^n(x)-t$ over $\F_q(t)$ and $\bar{\F}_q(t)$ respectively. We assume $\varphi^n(x)-t$ is separable so that $K_n/\F_q(t)$ is a Galois extension. This is a generic condition that is easy to verify; in fact it suffices to check that $\varphi'(x)\neq 0$. We also introduce the following definition. 

\begin{definition} Let $\Gamma$ be a finite set acting on a set $X$. We define the \textbf{fixed point proportion} of  $\Gamma$, denoted $\FPP(\Gamma)$, to be the proportion of elements in $\Gamma$ fixing an element of $X$. 
\end{definition}

We prove a generalization of equation (\ref{eqnBSD}) for iterates of rational functions in Section \ref{BoundsImage}. In Theorem \ref{FPP} we show
\begin{equation}\label{eqngeneralBSD}
\#\varphi^n\left(\bP^1(\F_q)\right) = \epsilon_nq+O\left([K_n:\F_q(t)]dnq^{1/2}\right).
\end{equation}
 As in equation (\ref{eqnBSD}),  $\epsilon_n=\FPP\left(\Gal\left(K_n/{\F}_q(t)\right)\right)$ if $K_n\cap\bar{\F}_q=\F_q$. More generally $\epsilon_n=\FPP\left(\sigma\Gal\left(K^*_n/\bar{\F}_q(t)\right)\right)$ where $\sigma\Gal\left(K^*_n/\bar{\F}_q(t)\right)$ is a generator of the factor group  $\Gal\big(K_n/{\F}_q(t)\big)\big/\Gal\left(K^*_n/\bar{\F}_q(t)\right)$. Unlike in equation (\ref{eqnBSD}), we see explicitly how the error term depends on the degree of $\varphi^n$, rather than including this information in the implied constant.  We take advantage of the iterated structure of $\varphi^n$ to refine the error term. 

In Section \ref{indicatrix}, \ref{indicatrixforgps},  and \ref{results}, we study $\epsilon_n$ in the case $\Gal\left(K^*_n/\bar{\F}_q(t)\right)$ is an iterated wreath product. If $q$ is large relative to $n$, this occurs under very general conditions \cite{JO, JKMT, Pink1}. 

In Section \ref{indicatrix}, we define indicatrix polynomials, our main tool for studying fixed point proportions, and give a preliminary result. In Section \ref{indicatrixforgps}, we perform some fairly involved calculations to obtain precise bounds on $\epsilon_n$ when $\Gal\left(K_n/\F_q(t)\right) = [G]^n$ for $G = C_d, S_d, A_d,$ or $D_d$. Here $[G]^n$ denotes the $n$-fold iterated wreath product of the group $G$ with itself, $S_d$, $A_d$, and $C_d$ denote the symmetric, alternating, and cyclic groups acting on $d$ letters, and $D_d$ denotes the group of symmetries of a regular $d$-gon. 

In Section \ref{results}, we combine our work in Sections \ref{indicatrix} and \ref{indicatrixforgps} with equation (\ref{eqngeneralBSD}) to get our main results, which are bounds on the image size. If $\varphi:\bP^1(\F_q)\rightarrow \bP^1(\F_q)$ is not a bijection, we see that $\epsilon_n$ is defined by a recursive formula and there is a constant $c_\varphi$ such that $\epsilon_n = \frac{2}{n c_\varphi} + O_d\left(\frac{\log n}{n^2}\right)$ for $n\geq 2$, see Theorem \ref{cor1}. Then we use our calculations from Section \ref{indicatrixforgps} to prove Theorem \ref{imbounds}.
 We also give, in Corollary \ref{prop:orbits}, a bound on the number of iterates that can occur before the critical orbits must either collide or cycle for degree $d$ rational functions $\varphi(x)\in \F_q(x)$ that are not bijections.

%

\begin{theorem}\label{imbounds} Suppose $K_n$ is tamely ramified over $\F_q(t)$ and  $\Gal\left(K_n/\F_q(t)\right) \cong [G]^n$, where $G=C_d, S_d,$ or $A_d$ for any $d\geq 2$. Then \[\#\varphi^n\left(\bP^1(\F_q)\right)=\epsilon_n q+O\left(|G|^{d^n}q^{1/2}\right),\] 
	where $\epsilon_n=\FPP([G]^n)$. Moreover,  $\epsilon_n=\frac{2}{nc_G}+O\left(\frac{\log n}{n^2}\right)$ if $n\geq 2$, where $c_{S_d}=c_{A_d} = 1$ and $c_{C_d}=d-1$.
\end{theorem}

In the case of $D_d$, we give the value of $c_{D_d}$, but do not prove a result as strong as Theorem \ref{imbounds} for this case. We include this case to work toward an application given in Theorem \ref{theorem.periodicbd}.

 Theorem \ref{imbounds} supports the model of random maps, which says a general rational function should behave like a random map. In \cite[Theorem 2]{FO}, Flajolet and Odlyzko show for a random map on a set with $q$ elements, the $n$-th iterate should have image size asymptotic to $(1-\tau_n)q$ as $q\rightarrow \infty$, where $\tau_0=0$ and $\tau_{n+1}=e^{-1+\tau_n}$. It is not hard to see that $1-\tau_n$ is asymptotic to $\frac{2}{n}$ as $n$ approaches infinity. As mentioned above, for large $q$ the Galois group of the splitting field of $f^n(x)-t$ over $\F_q(t)$ will be isomorphic to $[S_d]^n$ under very general conditions. The estimate in Theorem \ref{imbounds} gives the same asymptotic behavior as the heuristic predicts in this case.

In Sections \ref{Examples} and \ref{BoundsPer}, we turn our attention to polynomial maps defined over the ring of integers of a number field and the reduction of these maps to the residue fields. In Section \ref{Examples}, we focus on two families of polynomials defined over $\mathbb{Z}$, $\varphi(x) = ax^d+c$ and $\varphi(x) = (d-1)x^d+(da)x^{d-1}$. We obtain the following generalization of recent work of Shao \cite{Shao} and Heath-Brown \cite{HB}.

\begin{theorem}\label{example}
	Let $d>1$ and consider $\varphi(x)=ax^d+c\in \F_q(x)$, where $a,c\neq 0$. Suppose  $\varphi^n(0)\neq \varphi^m(0)$ for all $i<j\leq n$.
 If $q\equiv 1\mod d$, then  
		\[\#\varphi^n\left(\bP^1(\F_q)\right)=\epsilon_nq+O\left(d^{d^n}\sqrt{q}\right),\] 
		where $\epsilon_n=\frac{2}{(d-1)n}+O\left(\frac{\log n}{n^2}\right)$ for $n\geq 2$.
\end{theorem}

Shao shows for sufficiently large $p$ and $f(x)=x^2+1\in\F_p[x]$, we have $\#f^n\left(\F_p\right)=\mu_n p + O_n\left(\sqrt{p}\right)$, where $\mu_n$ is defined recursively by $\mu_0=1$ and $\mu_{n+1}=\mu_n-\frac{1}{2}\mu_n^2$ \cite[Theorem 1.6]{Shao}. Note, the implied constant in Shao's equation depends on the degree of $f^n(x)$, which is $2^n$.  More generally, Heath-Brown  shows if $\F_q$ is a finite field with odd characteristic and $f(x)=ax^2+bx+c\in \F_q[x]$ has the property that $f^i\left(-b/(2a)\right)\neq f^j\left(-b/(2a)\right)$ for all $0\leq i<j\leq n$, then $\#f^n(\F_q)=\mu_nq+O\left(2^{4^n}\sqrt{q}\right),$ with $\mu_n$ as above \cite{HB}. 
Theorem \ref{example} gives  $\#\varphi^n\left(\bP^1(\F_q)\right)=\epsilon_nq+O\left(2^{2^n}\sqrt{q}\right)$.  One can recover the recursive formula for $\mu_n$ from our formula for $\FPP\left([C_2]^n\right)$ in Section \ref{indicatrix}.

Finally in Section \ref{BoundsPer}, we find bounds on proportions of periodic points.
Let $K$ be a number field, $\varphi(x)$ be a polynomial with coefficients in $\mathcal{O}_K$, the ring of integers of $K$. Suppose the critical orbits of $\varphi$ are infinite and disjoint. Let $\varphi_\fp(x)$ be the reduction of $\varphi(x)$ modulo $\fp$ and consider the action of $\varphi_\fp:\bP^1\left(\cO_K/\fp\right)\rightarrow \bP^1\left(\cO_K/\fp\right)$. Let $q=\left|\mathcal{O}_K/\fp\right|$.
By \cite[Theorem 1.3(b)]{JKMT} and \cite[Theorem 3.8(b)]{Juul}, the proportion of periodic points approaches $0$ as $q$ approaches infinity. We obtain an explicit version of these results.

\begin{theorem}\label{theorem.periodicbd} Let $K$ be a number field, $\mathcal{O}_K$ the ring of integers of $K$, and $\varphi(x)\in \cO_K[x]$. Suppose $\{c\in \bar{k}:\varphi'(c)=0\}\subseteq \cO_K$ and for all $m,n\in\mathbb{N}$, $\varphi^n(a)\neq \varphi^m(b)$ for critical points $a,b$ unless $a= b$ and $n= m$.   Let $\fp$ be a prime of $\mathcal{O}_K$, $q=\left|\mathcal{O}_K/\fp\right|$, and $\varphi_\fp:\bP^1\left(\cO_K/\fp\right)\rightarrow \bP^1\left(\cO_K/\fp\right)$ be the reduction of $\varphi$ modulo $\fp$.
	\begin{enumerate}
		\item[(a)]  If each coset of $\Gal\left(K\left(\varphi^{-1}(t)\right)/K(t)\right)/\Gal\left(\bar{K}\left(\varphi^{-1}(t)\right)/\bar{K}(t)\right)$ contains at least one fixed point free element, then
		\[\frac{\#\Per(\varphi_\fp)}{q+1}=O_d \left(\frac{1}{\log\log q}\right).\]
		\item[(b)]  Suppose further that $K\left(\varphi^{-n}(t)\right)\cap \bar{K} = K$. There is a constant $A$, depending on $\varphi$ such that if $q\geq 2^{[K:\bQ]}e^{A}$ and $\Gal\left((\cO_K/\fp)\left(\varphi^{-1}(t)\right)/(\cO_K/\fp)(t)\right)=G$ is isomorphic to $S_d$ or $C_d$ for $d> 1$ or $A_d$ for $d> 5$. Then
		\[\frac{\#\Per(\varphi_\fp)}{q+1}< \frac{2\log d}{\log(\log q-[K:\bQ]\log 2)-\log A}+12 q^{-1/4}.\]
	\end{enumerate} 

\end{theorem}
A similar statement holds for $G=A_4$, we exclude $A_4$ here for ease of computation. 

Although Equation (\ref{eqngeneralBSD}) holds for any rational function defined over $\F_q$, the rest of this paper focuses on the case where $\Gal(K_n^*/\bar{\F}_q(t))$ is an iterated wreath product. There is current research focusing on what other groups can occur as Galois groups of these extentions, for example \cite{Pink1, BFH}. For the groups that appear, Equation (\ref{eqngeneralBSD}) can be used to bound the image size of iterates, given information about the fixed point proportion. The current method involving iterating an indicatrix polynomial will not directly apply, though it may be possible to find recursive formulas or bounds on the fixed point proportions for these other groups as in \cite{BFH}. Since this work relies on the Chebotarev Density Theorem, the methods in this paper can only be applied when we are able to find explicit bounds on the fixed point proportions of the Galois groups.


\section{Bounds on $\varphi^n(\bP^1(\F_q))$ in Terms of $\FPP(\Gal(K_n/\F_q(t)))$}\label{BoundsImage}

 Recall $K_n = \F_q(\varphi^{-n}(t))$. 
 In this section we prove the following.

\begin{theorem}\label{FPP} Let $\varphi(x)\in \F_q(x)$. Suppose $K_n/\F_q(t)$ is a tamely ramified extension and $K_n\cap\bar{\F}_q=\F_{q^r}$. Then 
\[\left|\frac{\#\varphi^n\left(\bP^1(\F_q)\right)}{\FPP\left(\sigma\Gal\left(K_n/\F_{q^r}(t)\right)\right)}-q\right|<M[K_n:\F_{q^r}(t)]ndq^{1/2}\]
	for some constant $M<6$, where $\sigma\in\Gal\left(K_n/\F_q(t)\right)$ is any element such that $\sigma|_{\F_{q^r}}=\Frob_q$ and $\sigma\Gal\left(K_n/\F_{q^r}(t)\right)$ denotes the coset of $\Gal\left(K_n/\F_{q^r}(t)\right)$ in $\Gal\left(K_n/\F_q(t)\right)$ containing $\sigma$.
	If the extension is geometric ($r=1$) this simplifies to 
	\[\left|\frac{\#\varphi^n\left(\bP^1(\F_q)\right)}{\FPP\left(\Gal(K_n/\F_{q}(t))\right)}-q\right|<M[K_n:\F_{q}(t)]ndq^{1/2}.\]
\end{theorem}


 Since $\Gal(K_n/\F_q(t))$ must be isomorphic to a subgroup of $[S_d]^n$, we see $[K_n:\F_q]d\leq d!^{\frac{d^n-1}{d-1}} d<d!^{d^n}$ \cite[Lemma 4.1]{odoni}. Note, this is a direct generalization of Birch and Swinnerton-Dyer's result mentioned in the introduction. The coefficient of $q$ in equation (\ref{eqnBSD}) is precisely $\FPP(S_d)$. This result applies to rational functions, not just polynomials. It also applies no matter what the Galois group is and regardless of whether or not the extension is geometric.

We use an effective version of the Chebotarev Density Theorem \cite[Proposition 6.4.8]{FJ}, which involves the genus of the extension and a count of the ramified primes. We make use of fact that we are working with an iterated function to get refined estimates for these quantities. 
We also use the following lemma, which has appeared several places in the literature, \cite{Cohen2, Jones3, Juul, JKMT}.

\begin{lemma}[\cite{Juul}, Lemma 3.6]\label{cccondition}
	Let $\alpha\in \bP^1(\F_q)$ such that $(t-\alpha)$ is unramified in $K_n$. Then $\alpha\in \varphi^n\left(\bP^1(\F_q)\right)$ if and only if elements of $\left(\frac{K_n/\F_q(t)}{t-\alpha}\right)$ fix some root of $\varphi^n(x)-t$, where $\left(\frac{K_n/\F_q(t)}{t-\alpha}\right)$ is the Frobenius conjugacy class of the primes of $K_n$ lying above $(t-\alpha)$,
\end{lemma}

\begin{proof}[Proof of Theorem~\ref{FPP}]
Fix $n$ and let $\mathcal{C}$ be a conjugacy class in $\Gal(K_n/\F_q(t))$. 
Let $C(K_n,\mathcal{C})$ denote the set of points $\alpha$ in $\bP^1(\F_q)$ for which the prime $(t-\alpha)$ is unramified in $K_n$ and  
$\left(\frac{K_n/\F_q(t)}{t-\alpha}\right)=\mathcal{C}$. 
Let $c$ denote the size of $\mathcal{C}$,
$g_{K_n}$ the genus of $K_n$, $K_n\cap \bar{\F}_q=\F_{q^r}$, and $m_n = [K_n:\F_{q^r}(t)]$. By Proposition 6.4.8 in \cite{FJ}, if  $\tau|_{\F_{q^r}}=\Frob_q$ for every $\tau\in \mathcal{C}$, then
\begin{equation}\label{cheb}
\left|\#C(K_n,\mathcal{C})-\frac{c}{m_n}q\right|<\frac{2c}{m_n}\left[(m_n+g_{K_n})q^{1/2}+m_n q^{1/4}+g_{K_n}+m_n\right],
\end{equation}
otherwise, $C(K_n,\mathcal{C})$ is empty.

Let $\mathcal{C'}$ denote the union of the conjugacy classes in $\Gal\left(K_n/\F_q(t)\right)$ fixing at least one root of $\varphi^n(x)-t$ and such that  $\tau|_{\F_q}=\Frob_q$ for all $\tau\in \mathcal{C}$. Let $c'$ denote the size of $\mathcal{C'}$ and let $C(K_n,\mathcal{C'})$ denote the set of points $\alpha$ in $\bP^1(\F_q)$ for which $(t-\alpha)$ is unramified in the splitting field $K_n$ of $\varphi^n(x)-t$ and $\left(\frac{K_n/\F_q(t)}{t-\alpha}\right)\subseteq \mathcal{C}'$. Summing the estimate given in equation (\ref{cheb}) over each conjugacy class in $\mathcal{C}'$ we see

\begin{align*}
\left|\#C(K_n,\mathcal{C'})-\frac{c'}{m_n}q\right|&=\left|\sum_{\mathcal{C}\subseteq \mathcal{C'}}\#C(K_n,\mathcal{C})-\sum_{\mathcal{C}\subseteq \mathcal{C'}}\frac{\#\mathcal{C}}{m_n}q\right|\\
&\leq\sum_{\mathcal{C}\subseteq \mathcal{C'}}\left|\#C(K_n,\mathcal{C})-\frac{\#\mathcal{C}}{m_n}q\right|\\
&<\sum_{\mathcal{C}\subseteq \mathcal{C'}}\frac{2\#\mathcal{C}}{m_n}\left[(m_n+g_{K_n})q^{1/2}+m_n q^{1/4}+g_{K_n}+m_n\right]\\
&=\frac{2c'}{m_n}\left[(m_n+g_{K_n})q^{1/2}+m_n q^{1/4}+g_{K_n}+m_n\right].\\
\end{align*}

Let $R_n$ denote the set of points ramifying in $K_n$. By Lemma \ref{cccondition}, $\left|\varphi^n\left(\bP^1(\F_q)\right)-\#C(K_n,\mathcal{C'})\right|\leq \#R_n$. Thus,
\begin{align}\left|\#\varphi^n\left(\bP^1(\F_q)\right)-\frac{c'}{m_n}q\right| 
&= \left|\#\varphi^n\left(\bP^1(\F_q)\right)-\#C\left(K_n,\mathcal{C'}\right)+\#C\left(K_n,\mathcal{C'}\right)-\frac{c'}{m_n}q\right|\nonumber\\
&\leq \left|\#\varphi^n\left(\bP^1(\F_q)\right)-\#C\left(K_n,\mathcal{C'}\right)\right|+\left|\#C\left(K_n,\mathcal{C'}\right)-\frac{c'}{m_n}q\right|\nonumber\\
&<\#R_n+\frac{2c'}{m_n}\left[(m_n+g_{K_n})q^{1/2}+m_n q^{1/4}+g_{K_n}+m_n\right]. \label{estimate1}
\end{align}

Using a discriminant argument and induction, we can see that any prime ramifying in $K_n$ has the form $\varphi^n(a)-t$ where $a$ is a critical point of $\varphi$ (see \cite[Proposition 1]{CH} or \cite[Lemma 3.4]{JKMT}). Hence, there are at most $n(2d-2)$ ramified primes. 
By the Riemann-Hurwitz formula we have
\begin{align*}2g_{K_n}-2&=-2m_n+ \sum_{\fp\in\mathbb{P}_{\F_q(t)}}\sum_{\fp'|\fp} \left(e(\fp'|\fp)-1\right)\deg\fp'\\
2g_{K_n}-2&\leq -2m_n+n(2d-2)(m_n-1)\\
g_{K_n}&\leq (m_n-1)(nd-n-1),
\end{align*}
where the upper bound on the sum in the first line is the maximum number of ramified primes times an upper bound on the value of the inner sum.

Plugging these bounds into equation (\ref{estimate1}) and simplifying, we see
\[\left|\#\varphi^n\left(\bP^1(\F_q)\right)-\frac{c'}{m_n}q\right|
<M_0\frac{2c'}{m_n}m_n nd q^{1/2}\]
for a constant $M_0<3$. 
Since $\frac{c'}{m_n}= \FPP(\sigma\Gal(K_n/\F_{q^r}(t)))$ where $\sigma \in \Gal(K_n/\F_q(t))$ is any element such that $\sigma|_{\F_{q^r}}=\Frob_q$ and $m_n=[K_n:\F_{q^r}(t)]$, we have 
\[\left|\frac{\#\varphi^n\left(\bP^1(\F_q)\right)}{\FPP\left(\sigma\Gal(K_n/\F_{q^r}(t))\right)}-q\right|<2M_0\left[K_n:\F_{q^r}(t)\right]ndq^{1/2},\]
which completes the proof of Theorem \ref{FPP}.
\end{proof}


\section{Indicatrix Polynomials, Wreath Products, and Fixed Point Proportions}\label{indicatrix}

The indicatrix polynomial of a set of permutations carries information about how many elements of the set fix each possible number of points. This function was developed by Polya \cite{Polya}, generalized by Harary and Palmer \cite{HP}, and used in this current context by Odoni \cite{odoni}.

\begin{definition}
	Let $\Gamma$ be a finite set of permutations acting on a finite set $X$. The \textbf{indicatrix} of $\Gamma$ is the polynomial \[\Phi_\Gamma (x) = \frac{1}{\#\Gamma}\sum_{\gamma\in\Gamma} x^{\tr \gamma},\] where $\tr \gamma$ is the number of points of $X$ fixed by $\gamma$.
\end{definition}

For any set $\Gamma$,  $\Phi_\Gamma$ is a polynomial of degree at most $|X|$ and $\Phi_\Gamma(1)=1$. The coefficient of $x^i$ in $\Phi_\Gamma$ is the proportion of $\gamma\in\Gamma$ with exactly $i$ fixed points. In particular, the constant term of $\Phi_\Gamma$ is the proportion of $\gamma$ in $\Gamma$ with no fixed points, so $\FPP(\Gamma)=1-\Phi_\Gamma(0)$. 


\begin{lemma}[\cite{odoni}, Lemma 4.2]\label{lemma:odoni_indicatrix} For permutation groups $G$ and $H$, acting on sets $X$ and $Y$ respectively, the wreath product $G[H]$ has a natural action on the set $X\times Y$ and the indicatrix function satisfies  
	\[\Phi_{G[H]}(x)=\Phi_G\circ\Phi_H (x).\]
\end{lemma}

If a group $G$ acts on a set $X$ then the $n$-fold iterated wreath product $[G]^n$ acts on $X^n$. It follows immediately from Lemma~\ref{lemma:odoni_indicatrix} that $\Phi_{[G]^n}(x)=\Phi_G^n(x)$ and $\FPP([G]^n)=1-\Phi_G^n(0)$. Odoni uses this to show  $\FPP\left([G]^n\right) = \frac{2}{n\Phi_G''(1)}\left(1+O_G\left(\frac{\log n}{n}\right)\right)$  where $O_G$ is a constant depending only on the group $G$ \cite[Lemma 4.3]{odoni}.  
We extend this idea in the following lemma. We will see in the Section \ref{results} we can apply this lemma in the case $\Gal(K^*_n/\bar{F}_q(t))\cong[G]^n$. The proof is similar to \cite[Lemma 4.3]{odoni},  we include the details here for completeness.

\begin{lemma}\label{cosetasymptotic} Let $G$ be a transitive subgroup of $S_d$ and $\tau$ an element of $S_d$ such that the coset $\tau G$ has at least one element with no fixed points. Then \[1-\Phi^n_{\tau G}(0)=\frac{2}{n\Phi_{\tau G}''(1)}+O_{d}\left(\frac{\log n}{n}\right) \text{ for $n\geq 2$}.\]
\end{lemma}

\begin{proof}
	
	Let $\Phi(x) = \Phi_{\tau G}(x)$ and $X=\{1,2,\dots, n\}$.
	We start by showing $\lim_{n\rightarrow \infty} \Phi^n(0)=1$. Since $G$ is transitive, for any $x\in X$ we can find an $h\in G$ such that $\tau h (x) = x$, then
	 $|\Stab_{\tau G}(x)| =|\{ g\in G: \tau g(x)=x\}| = |\{g\in G: \tau g(x) =\tau h(x)\}|=|\{g\in G: h^{-1} g(x) =x\}| = |\Stab_G(x)|$. Thus, using the orbit/stabilizer theorem, we see
	\begin{align*} \Phi'(1)&=\frac{1}{|\tau G|}\sum_{g\in G}\tr(\tau g)=\frac{1}{|\tau G|}\sum_{x\in X}|\Stab_{\tau G}(x)|=\frac{1}{|G|}\sum_{x\in X}|\Stab_{G}(x)|\\&=\frac{1}{|G|}\sum_{x\in X}\frac{|G|}{|\Orb_G(x)|}=\sum_{x\in X}\frac{1}{|X|}=1.\end{align*}
		
	Since $\Phi'(1)=1$, the line $x=y$ is tangent to the graph of $\Phi(x)$ at $x=1$. 
	Also, since the coefficients of $\Phi$ are nonnegative, $\Phi'(x),\Phi''(x)\geq 0$ for all $x\in[0,1]$. We assumed $\tau G$ has at least one element with no fixed points, so $\Phi(0)>0$. This implies the graph of $\Phi(x)$ lies above the tangent line $x=y$, that is $\Phi(x)> x$ for all $x\in [0,1)$, and hence the sequence $\{\Phi^n(0)\}$ is strictly increasing. The sequence is also bounded above by $1$, so it must converge. The limit must be a fixed point of $\Phi(x)$, hence the limit is $1$.   
	
	Now let $\epsilon_n= 1-\Phi^n(0)$  and $E_n=\frac{1}{\epsilon_n}$, so $\epsilon_n$ is strictly decreasing to $0$ and $E_n$ is strictly increasing to $\infty$. Note, \[1-\epsilon_{n+1}=\Phi(1-\epsilon_n)=\Phi(1)-\Phi'(1)\epsilon_n+\frac{\Phi''(1)}{2}\epsilon_n^2+\dots+\frac{(-1)^d\Phi^{(d)}(1)}{d!}\epsilon_n^d.\]
	Using the fact that $\Phi(1)=1$ and $\Phi'(1)=1$, we have
	\[\epsilon_{n+1}= \epsilon_n-\frac{\Phi''(1)}{2}\epsilon_n^2+\dots+\frac{(-1)^{d-1}\Phi^{(d)}(1)}{d!}\epsilon_n^d\] and hence
	\begin{align}E_{n+1}&=\frac{1}{\epsilon_{n+1}}=\frac{1}{\epsilon_n-\frac{\Phi''(1)}{2}\epsilon_n^2+\dots+\frac{(-1)^{d-1}\Phi^{(d)}(1)}{d!}\epsilon_n^d} \nonumber \\
	&=\frac{E_n^d}{E_n^{d-1}-\frac{\Phi''(1)}{2}E_n^{d-2}+\dots+\frac{(-1)^{d-1}\Phi^{(d)}(1)}{d!}}\nonumber \\
	&=E_n+\frac{\Phi''(1)}{2}+\frac{f(E_n)}{g(E_n)}, \label{eqn:rec}
	\end{align}
	where $f(x)$ is a polynomial of degree less than or equal to $d-2$ and $g(x)$ is a polynomial of degree $d-1$. Further, $g(x)=x^d(1-\Phi(1-\frac{1}{x}))$, so $g(x)>0$ for all $x>1$. This implies $\frac{xf(x)}{g(x)}$ is continuous on $(1,\infty)$. Since $\deg(xf(x))\leq \deg(g(x))$, the limit $\lim_{x\rightarrow \infty}\frac{xf(x)}{g(x)}$ is finite. Hence $\frac{xf(x)}{g(x)}$ is bounded on $(1,\infty)$, and so for all $x>1$ we have  $\left|\frac{f(x)}{g(x)}\right|<\frac{C_1}{x}$ for some constant $C_1$.
	Applying equation (\ref{eqn:rec}) recursively, we see  
	\begin{equation}\label{eqn:E_n}
	\left|E_{n+1}-\frac{(n+1)\Phi''(1)}{2}\right|< C_1\sum_{i=1}^nE_i^{-1}. 
	\end{equation}
	
	
	On the other hand, we have $E_{n+1}\geq E_n+\frac{\Phi''(1)}{2}-C_{1}E_n^{-1}$. Since $E_n$ is strictly increasing to infinity as $n\rightarrow \infty$, there is some $N$ depending on $\Phi$ such that $E_n\geq \frac{4C_1}{\Phi''(1)}$ for all $n\geq N$. Thus, $E_n\geq E_{N}+(n-N)\frac{\Phi''(1)}{4}$ for all $n\geq N$. It follows that $E_n\geq C_2n$  for some constant $C_2$ depending on $\Phi$.
	Substituting into equation~(\ref{eqn:E_n}) and setting $C=\frac{C_1}{C_2}$, we have 
	\[\left|E_{n+1}-\frac{(n+1)\Phi''(1)}{2}\right|\leq \frac{C_1}{C_2}\left(\sum_{i=1}^n\frac{1}{i}\right)\leq C\log (n+1).\] 
	Therefore, $\left|E_{n}-\frac{n\Phi''(1)}{2}\right|\leq C\log n$ for $n\geq 2$. From this we can see \[\epsilon_n=\frac{2}{n\Phi''(1)}+O_{\tau G}\left(\frac{\log n}{n^2}\right).\] 
	Note, the constants $C_1, C_2$ depend on the polynomial $\Phi$ and hence on $\tau G$.
		There are only finitely many subgroups $G$ of $S_d$ and for each $G$ only finitely many choices of $\tau \in S_d$ so that $\tau G$ has a fixed point. Hence, taking the maximum of the implied constants over all choices for $\tau$ and $G$ in $S_d$, we can write
	\[\epsilon_n=\frac{2}{n\Phi''(1)}+O_{d}\left(\frac{\log n}{n^2}\right) \text{ for $n\geq 2$}.\] 
\end{proof}

\section{Bounds for Wreath Products of Cyclic, Symmetric, Alternating, and Dihedral Groups}\label{indicatrixforgps}	

The following lemma will be useful throughout this section.

\begin{lemma}\label{ncompare}
	Let $\Phi(x)$ and $\Psi(x)$ be increasing functions on an interval $I$. If $\Phi(x)\leq \Psi(x)$ for all $x\in I$ then, \begin{align*}\Phi^n(x)&\leq \Psi^n(x),  
	\end{align*} 
	for all $x\in I$.
\end{lemma}

\begin{proof} This follows by induction on $n$. We have
	$\Phi^{n+1}(x)=\Phi(\Phi^n(x))\leq\Psi(\Phi^n(x))\leq\Psi(\Psi^n(x))=\Psi^{n+1}(x),$
	where the first inequality follows from the $n=1$ case and the second follows from the induction hypothesis and the fact that $\Psi(x)$ is an increasing function. 
\end{proof}

\subsection{Fixed point proportion for $[C_d]^n$}

Let $\Phi_d(x)$ denote the indicatrix of $C_d$, the cyclic group with $d$ elements. Then $\Phi_d^n(x)$ is the indicatrix of $[C_d]^n$. Note, \[\Phi_d(x)=\frac{x^d}{d}+\frac{d-1}{d}.\]

\begin{proposition}\label{Cbound}
	For all $d\geq 2$, \[\frac{2}{(d-1)(n + 4 + \log(n))}< \FPP([C_d]^n) <  \frac{2}{(d-1)(n+1)}. \]
\end{proposition}

\begin{proof} Fix $d$. Let $a_n=\Phi_d^n(0)$
	and $b_n=\frac{2}{(d-1)(1-a_n)}$. We will show $n+1 < b_n< n+4+\log(n)$. First, $b_1=\frac{2d}{d-1}$ so $2< b_1\leq 4$. We have
	\begin{align*}
	b_{n+1}&=\frac{\frac{2}{d-1}}{1-a_{n+1}}=\frac{\frac{2}{d-1}}{1-\left(\frac{a_n^d}{d}+\frac{d-1}{d}\right)} = 
	\frac{\frac{2}{d-1}d}{1-a_n^d}
	=\frac{\frac{2}{d-1}d}{1-\left(1-\frac{2}{(d-1)b_n}\right)^d}=\frac{\frac{2}{d-1}db_n^d}{b_n^d-\left(b_n-\frac{2}{(d-1)}\right)^d}\\
	&= b_n+1+\frac{\frac{2}{d-1}db_n^d-(b_n+1)\left(b_n^d-\left(b_n-\frac{2}{(d-1)}\right)^d\right)}{b_n^d-\left(b_n-\frac{2}{(d-1)}\right)^d}.
	\end{align*} 
	We look at the numerator and denominator of the fraction more closely. First note the denominator simplifies as 
	\[\sum_{i=1}^d (-1)^{i+1}{d\choose i} b_n^{d-i}\left(\frac{2}{d-1}\right)^i.\]
	This is an alternating sum with decreasing terms, hence it is greater than the sum of the first two terms $\frac{2d}{d-1}(b_n^{d-1}-b_n^{d-2})$.
	The numerator is 
	\begin{align*}
	&\frac{2}{d-1}db_n^d-(b_n+1)\left(\sum_{i=1}^d (-1)^{i+1}{d\choose i} b_n^{d-i}\left(\frac{2}{d-1}\right)^i\right)\\
	&=\frac{2}{d-1}db_n^d+\left(\sum_{i=1}^d (-1)^{i}{d\choose i} b_n^{d+1-i}\left(\frac{2}{d-1}\right)^i\right)+\left(\sum_{i=1}^d (-1)^{i}{d\choose i} b_n^{d-i}\left(\frac{2}{d-1}\right)^i\right)\\
	&=\left(\sum_{i=1}^{d-1} (-1)^{i+1}{d\choose {i+1}} b_n^{d-i}\left(\frac{2}{d-1}\right)^{i+1}\right)+\left(\sum_{i=1}^d (-1)^{i}{d\choose i} b_n^{d-i}\left(\frac{2}{d-1}\right)^i\right)\\
	&=(-1)^d\left(\frac{2}{d-1}\right)^d+\sum_{i=1}^{d-1}(-1)^i\left(\frac{2}{d-1}\right)^i\left({d\choose i}-{d\choose i+1}\frac{2}{d-1}\right)b_n^{d-i}\\
	&=\sum_{i=2}^{d}(-1)^i\left(\frac{2}{d-1}\right)^i{d+1\choose i+1}\frac{i-1}{d-1}b_n^{d-i}.
	\end{align*}
	Note, this is an alternating sum with decreasing terms. Hence, the numerator is positive and is less than \[\left(\frac{2}{d-1}\right)^2{d+1\choose 3}\frac{1}{d-1}b_n^{d-2}=\frac{2d(d+1)}{3(d-1)^2}b_n^{d-2}.\]
	Since the fraction is positive, we see that $b_{n+1}>b_n+1$ for all $n$ and hence by induction, $b_n>n+1$ for all $n$. 
	Also,
	\begin{align*}b_{n+1}
	&<b_n+1+\frac{\frac{2d(d+1)}{3(d-1)^2}b_n^{d-2}}{\frac{2d}{d-1}(b_n^{d-1}-b_n^{d-2})}\\&=b_n+1+\frac{d+1}{3(d-1)}\frac{1}{b_n-1}\\&<b_n+1+\frac{d+1}{3(d-1)}\frac{1}{n}\\
	&\leq b_n+1+\frac{1}{n}.
	\end{align*}
	Thus, we see
	\begin{align*}
	b_{n+1}&<b_1+n+\sum_{i=1}^{n}\frac{1}{i}
	<4+n+1+\log(n)
	<(n+1)+4+\log(n+1).
	\end{align*}
	
\end{proof}

\subsection{Fixed point proportion for $[S_d]^n$}

Let $\Phi_d(x)$ denote the indicatrix of $S_d$, the symmetric group on $d$ letters.

\begin{lemma}\label{Sindicatrix}
	For $d\geq 2$, \[\Phi_d(x)=\sum_{j=0}^{d} \frac{\Phi_{d-j}(0)x^j}{j!},\] where $\Phi_k(0)=\sum_{i=0}^k \frac{(-1)^i}{i!}$. Thus, $\Phi_d'(x)=\Phi_{d-1}(x)$.
\end{lemma}

\begin{proof}

	The coefficient of $x^j$ is the proportion of elements of $S_d$ fixing $j$ letters. Note, there are ${d\choose j}$ choices for the $j$ fixed letters, and then $(d-j)!\Phi_{d-j}(0)$ ways to permute the remaining letters so that none are fixed. Thus, the coefficient is $\frac{{{d}\choose{j}}(d-j)! \Phi_{d-j}(0)}{d!}=\frac{\Phi_{d-j}(0)}{j!}$, and further \[\Phi_d'(x)=\sum_{j=1}^{d} \frac{j\Phi_{d-j}(0)x^{j-1}}{j!}=\sum_{j=0}^{d-1} \frac{\Phi_{d-1-j}(0)x^{j}}{j!}=\Phi_{d-1}(x).\]
	
	The formula for $\Phi_k(0)$ is given in \cite[Corollary 2.6]{BDF}.
\end{proof}

\begin{lemma}\label{Scompare} Let $k\geq 2$. Then for all $d>k$, \begin{align*}\Phi_{d}(x)&\leq \Phi_{k}(x) \text{ if $k$ is even }  \\  \Phi_{d}(x)&\geq \Phi_{k}(x) \text{ if $k$ is odd } 
	\end{align*}  on the interval $[0,1]$ with equality  only when $x=1$.
\end{lemma}

\begin{proof}
	
	We let $\Phi_1(x)=x$ and we let $\Phi_0(x)=1$. 
	It is clear from the definition of the indicatrix that $\Phi_d(1)=1$ for any $d$. Also, since $\Phi_j(0)=\sum_{i=0}^j \frac{(-1)^i}{i!}$, we see that these inequalities hold when $x=0$. 
	
	Now we proceed by induction on $k$, starting with the cases $k=0$ and $k=1$.
	Note, $\Phi_d(x)<1=\Phi_0(x)$ for all $x\in [0,1)$. Also, since \[(\Phi_d(x)-x)'=\Phi_{d-1}(x)-1<0\] on $[0,1)$, the function $\Phi_d(x)-x$ is decreasing to $0$ as $x$ approaches $1$, and hence is positive on this interval. Thus, $\Phi_d(x)>x=\Phi_1(x)$ for all $d>1$.

	If $k$ is even and $d>k$, then $k-1$ is odd and by the induction hypotheses \[(\Phi_d(x)-\Phi_k(x))'=\Phi_{d-1}(x)-\Phi_{k-1}(x)>0,\] so $\Phi_d(x)-\Phi_k(x)$ is increasing to $0$ and hence $\Phi_d(x)<\Phi_k(x)$ on $[0,1)$. Similarly, if $k$ is odd and $d>k$, then  \[(\Phi_d(x)-\Phi_k(x))'=\Phi_{d-1}(x)-\Phi_{k-1}(x)<0, \] and hence $\Phi_d(x)>\Phi_k(x)$ on $[0,1)$.
	
\end{proof}

\begin{proposition}\label{Sbound}
	For all $d\geq 2$, \[\frac{2}{n+4+\log(n)}< \FPP([S_d]^n)\leq\frac{2}{n+2}.\]
\end{proposition}

\begin{proof}
	By Proposition~\ref{Scompare}, $\Phi_d(x)\leq \Phi_2(x)$ for all $d$. So the lower bound follows from Proposition \ref{Cbound} and Lemma \ref{ncompare} since $C_2=S_2$.
	
	For the upper bound, Proposition~\ref{Scompare} implies $\Phi_d(x)\geq \Phi_3(x)$ for all $d$.  By Lemma \ref{ncompare}, it suffices to show $\Phi_3^n(0)\geq 1-\frac{2}{n+2}$ for all $n$. Let $a_{n}=\Phi_3^n(0)$. 
	Note, $a_{1}=1/3= 1-2/3$. If $a_{n}\geq 1-2/(n +2)$ for some $n\geq 1$,  then

	\begin{align*}
	a_{n+1}& =\frac{1}{6}a_{n}^3+\frac{1}{2}a_{n}+\frac{1}{3}\\
	& \geq\frac{1}{6}\left(1-\frac{2}{n+2}\right)^3+\frac{1}{2}\left(1-\frac{2}{n+2}\right)+\frac{1}{3}\\
	& =\frac{6n^3+24n^2+36n+16}{6n^3+36n^2+72n+48}\\
	& =1-\frac{12n^2+36n+32}{6n^3+36n^2+72n+48}\\
	& > 1-\frac{2}{n+3}.
	\end{align*}
	
\end{proof}

\subsection{Fixed point proportion for $[A_d]^n$}

Let $\Phi_d(x)$ denote the indicatrix of $A_d$, the alternating group on $d$ letters.

\begin{lemma}\label{Aindicatrix}
	For $d\geq 2$, \[\Phi_d(x)=\sum_{j=0}^{d} \frac{\Phi_{d-j}(0)x^j}{j!},\] where $\Phi_k(0)=(-1)^{k-1}\frac{k-1}{k!}+\sum_{i=0}^k \frac{(-1)^i}{i!} $. Thus, $\Phi_d'(x)=\Phi_{d-1}(x)$.
\end{lemma}

\begin{proof}

	Note, there are ${d}\choose{j}$ choices for the $j$ fixed letters, and then $\frac{(d-j)!}{2}\Phi_{d-j}(0)$ ways to permute the remaining letters so that none are fixed and the permutation is even. Thus, the coefficient is $\frac{{{d}\choose{j}}\frac{(d-j)!}{2} \Phi_{d-j}(0)}{d!/2}=\frac{\Phi_{d-j}(0)}{j!}$. 
	
	The formula for $\Phi_k(0)$ is given in \cite[Corollary 2.6]{BDF}.
\end{proof}	

\begin{lemma}\label{Acompare} Let $k\geq 3$ and $d>k$. Then, \begin{align*}\Phi_{d}(x)&\geq \Phi_{k}(x) \text{ if $k$ is even }  \\  \Phi_{d}(x)&\leq \Phi_{k}(x) \text{ if $k$ is odd } 
	\end{align*}  on the interval $[0,1]$ with equality  only when $x=1$.
\end{lemma}

\begin{proof}
	We will prove this by showing when $k$ is even $\Phi_{k+2}(x)\geq \Phi_k(x)$ and $\Phi_{k+1}(x)\geq \Phi_k(x)$ for $x\in [0,1)$ and when $k$ is odd $\Phi_{k+2}(x)\leq \Phi_k(x)$ and $\Phi_{k+1}(x)\leq \Phi_k(x)$ for $x\in [0,1)$. It is clear from the definition of the indicatrix that $\Phi_d(1)=1$ for any $d$. 
	
	It is easy to check $\Phi_5(x) - \Phi_3(x)= \left(  \frac{x^5}{60}+\frac{x^2}{3}+\frac{x}{4}+\frac{2}{5}\right)-\left(\frac{x^3}{3}+\frac{2}{3}\right)<0$ and $\Phi_4(x)-\Phi_3(x) = \left(\frac{x^4}{12}+\frac{2x}{3} +\frac{1}{4}\right) -\left(\frac{x^3}{3}+\frac{2}{3}\right)<0 $ for all $x \in [0,1)$, establishing the result in these cases. 
	
	If $k$ is even, then  $(\Phi_{k+2}(x)-\Phi_k(x))' = \Phi_{k+1}(x)-\Phi_{k-1}(x)$ and $(\Phi_{k+1}(x)-\Phi_k(x))' = \Phi_{k}(x)-\Phi_{k-1}(x)$, by induction on $k$, each of these will be negative for $x\in [0,1)$. This implies that $\Phi_{k+2}(x)-\Phi_k(x)$ and $\Phi_{k+1}(x)-\Phi_k(x)$ are decreasing to $0$ on $[0,1)$ and hence are positive on this interval. Similarly, if $k$ is odd, each of these derivatives will be positive and hence increasing to $0$ on $[0,1)$, so they negative on this interval. 
\end{proof}

\begin{proposition}\label{Abound}
	For all $d\geq 5$, 
	\[\frac{2}{n+4+\log(n)}<\FPP([A_d]^n)<\frac{2}{n+2}\] and $\frac{2}{n+2}<\FPP([A_4]^n)<\frac{2}{n+1-\log(n)}.$
\end{proposition}

\begin{proof}
	It is easy to check that $\Phi_{S_3}(x)\leq \Phi_{A_6}(x)$ and $\Phi_{A_5}(x)\leq\Phi_{S_2}(x)$ for all $x\in [0,1]$. Hence, by Lemma \ref{Acompare} $\Phi_{S_3}(x)\leq \Phi_{A_d}(x)\leq \Phi_{S_2}(x)$ for all $d\geq 5$. Then the result for $d\geq 5$ follows from Lemma \ref{ncompare} and Proposition \ref{Sbound}.
	
	Now, for $d=4$, write $a_n=\Phi_4^n(0)$ and let $b_n=\frac{2}{1-a_n}$, we will show $n+1-\log(n) < b_n<n+2$. Note, $2<b_1=\frac{8}{3}<3$. For $n\geq 2$, we have
	\begin{align*}
	b_{n+1}&=\frac{2}{1-a_{n+1}}
	=\frac{2}{1-\left(\frac{a_n^4}{12}+\frac{2a_n}{3}+\frac{1}{4}\right)} 
	=	\frac{24}{9-a_n^4-8a_n}\\
	& = \frac{24}{9-\left(1-\frac{2}{b_n}\right)^4-8\left(1-\frac{2}{b_n}\right)}
	= \frac{24b_n^4}{24b_n^3-24b_n^2+32b_n-16}\\
	&=b_n+1-\frac{b_n^2+2b_n-2}{3b_n^3-3b_n^2+4b_n-2}.
	\end{align*} 
	We can see $\frac{x^2+2x-2}{3x^2-3x+4-2/x}$ is positive and decreasing for $x\geq 2$. Since $b_n>2$ for all $n$, we have $b_{n+1}<b_n+1<n+2$ and $\frac{b_n^2+2b_n-2}{3b_n^3-3b_n^2+4b_n-2}=\frac{1}{b_n}\frac{b_n^2+2b_n-2}{3b_n^2-3b_n+4-\frac{2}{b_n}}<\frac{2}{3b_n}$. Hence
	\[b_{n+1}>b_1+{n}-\frac{2}{3}\sum_{i=1}^{n}\frac{1}{b_i}>n+\frac{8}{3}-\frac{2}{3}\sum_{i=1}^n\frac{1}{b_i}.\]
	
	Now we show $b_n>n+1-\log n$ by induction. We have
	\begin{align*}
	\sum_{i=1}^n\frac{1}{b_i}&<\sum_{i=1}^n\frac{1}{i+1-\log i}=\sum_{i=1}^n\frac{1}{i+1}+\sum_{i=1}^n\frac{\log i}{(i+1)(i+1-\log i)}\\
	&<\int_1^{n+1}\frac{1}{x} \; dx+\int_1^\infty \frac{\log x}{(x+1)(x+1-\log x)}\; dx < \log(n+1)+1. 
	\end{align*}
	Therefore, $b_{n+1}>n+\frac{8}{3}-\frac{2}{3}(\log(n+1)+1) >n +2 -\log(n+1)$. 
	
\end{proof}

\subsection{Fixed point proportion for $[D_d]^n$}

Let $d\geq3$ and let $\Phi_d(x)$ denote the indicatrix of $D_d$, the dihedral group of symmetries on a regular polygon with $d$ sides.
If $d$ is odd, then \[\Phi_d(x)=\frac{1}{2d}\left(x^d+dx+(d-1)\right),\] and if $d$ is even, then \[\Phi_d(x)=\frac{1}{2d}\left(x^d+\frac{d}{2}x^2+\frac{3d-2}{2}\right).\]

\begin{lemma}\label{Dcompare} Let $d>k\geq 2$. If $k, d$ are both even or both odd then \[\Phi_{d}(x)\geq \Phi_{k}(x) \]  on the interval $[0,1]$ with equality  only when $x=1$ and if $k$ is odd \[\Phi_{k+1}(x)\geq \Phi_{k}(x).\] 
\end{lemma}

\begin{proof}
	It is clear from the definition of the indicatrix that $\Phi_d(1)=1$ for any $d$. 
	
	Let $d>k\geq 2$  with $k,d$ even. Then
	
	\[(\Phi_d(x)-\Phi_k(x))'=\left(\frac{x^{d-1}}{2}+\frac{x}{2}\right)-\left(\frac{x^{k-1}}{2}+\frac{x}{2}\right)=\frac{1}{2}(x^{d-1}-x^{k-1})<0\]
	on $[0,1)$. Thus, the function $\Phi_{d}(x)-\Phi_k(x)$ is decreasing to $0$ as $x$ approaches $1$, and hence is positive on this interval and $\Phi_{d}(x)>\Phi_k(x)$.
	
	Similarly, if $d>k\geq 2$  with $k,d$ odd, then
	\[(\Phi_d(x)-\Phi_k(x))'=\left(\frac{x^{d-1}}{2}+\frac{1}{2}\right)-\left(\frac{x^{k-1}}{2}+\frac{1}{2}\right)=\frac{1}{2}(x^{d-1}-x^{k-1})<0\] on $[0,1)$ and $\Phi_{d}(x)>\Phi_k(x)$.
	
	Finally, if $k$ is odd then $\Phi_{k+1}(0)>\Phi_k(0)$ and  \[(\Phi_{k+1}(x)-\Phi_k(x))'=\left(\frac{x^k}{2}+\frac{x}{2}\right)-\left(\frac{x^{k-1}}{2}+\frac{1}{2}\right)=\frac{x^{k-1}}{2}(x-1)+\frac{1}{2}(x-1)<0\] on $[0,1)$. So the function $\Phi_{k+1}(x)-\Phi_k(x)$ is decreasing to $0$ as $x$ approaches $1$, and hence is positive on this interval. Thus, $\Phi_{k+1}(x)>\Phi_k(x)$ for all odd $k\geq 3$. 
\end{proof}

\begin{proposition}\label{Dbound}
	For all  $d\geq 3$,  \[\FPP([D_d]^n)<\frac{2}{n+2}.\]
\end{proposition}

\begin{proof}
	By Lemmas \ref{Dcompare} and \ref{ncompare}, it suffices to show $\FPP([D_3]^n)<\frac{2}{n+2}$. This follows from Proposition \ref{Sbound} since $D_3=S_3$.
	
\end{proof}


\section{Bounds on Image Size}\label{results}

In this section we use the results of the previous sections to give bounds on the image size of iterates of a rational map $\varphi$ when the Galois group of the geometric part of the splitting field extension is an iterated wreath product.

\begin{theorem}\label{cor1}
	Suppose $K_n/\F_q(t)$ is tamely ramified, 
	$\Gal(K_n^*/\bar{\F}_q(t))\cong [G]^n$ for some group $G$, and $\varphi:\bP^1(\F_q)\rightarrow\bP^1(\F_q)$ is not a bijection. Then \[\#\varphi^n(\bP^1(\F_q)) =\epsilon_n q+O_{d}\left(|G|^{d^n}q^{1/2}\right),\]
	 where  $\epsilon_n=\FPP(\sigma\Gal(K_n/\F_{q^r}(t)))$ and $\F_{q^r}=K_1\cap\bar{\F}_q$ for any $\sigma\in \Gal(K_n/\F_q(t))$ such that $\sigma|_{{\F}_{q^r}}=\Frob_q$. Moreover, $\epsilon_1=1-\Phi(0)$ and $\epsilon_n=\frac{2}{n\Phi''(1)}+O_d\left(\frac{\log n}{n^2}\right)$ for $n\geq 2$, where $\Phi(x)$ is the indicatrix function for  the coset $(\sigma|_{K_1}) \Gal(K_1/\F_{q^r}(t))$. 
\end{theorem}




\begin{proof}

	Let $\F_{q^r}=K_1\cap \bar{\F}_q$. Then $\Gal(K_1/\F_{q^r}(t))\cong G$ and  \[[G]^n\cong \Gal(K_n^*/\bar{\F}_q(t))\subseteq \Gal(K_n/\F_{q^r}(t))\subseteq [G]^n\]  by \cite[Lemma 3.3]{JKMT}. Hence,  we must have $\Gal(K_n/\F_{q^r}(t))\cong[G]^n$ and $K_n\cap\bar{\F}_q=\F_{q^r}$ by \cite[Proposition 3.6]{JKMT}. 
	
	Let $\sigma$ be any element of $\Gal(K_n/\F_q(t))$ such that $\sigma|_{\F_{q^r}}=\Frob_q$.  Let $\Phi(x)$ be the indicatrix function for the coset  $(\sigma|_{K_1}) \Gal(K_1/\F_{q^r}(t))$ in $\Gal(K_1/\F_{q}(t))$. Then by \cite[Lemma 2.2]{Juul}, $\Phi_{\sigma\Gal(K_n/\F_{q^r}(t))} = \Phi^n$.	
	Since we assumed $\varphi:\F_q\rightarrow \F_q$ is not a bijection, at least one element of $(\sigma|_{K_1}) \Gal(K_1/\F_{q^r}(t))$ has no fixed points by \cite[Lemma 4.3 and Proposition 4.4]{GTZ}. 	
	Let $\epsilon_n=\FPP(\sigma\Gal(K_n/{\F}_{q^r}(t)))=1-\Phi^n(0)$.  Note, $d[K_n:\F_{q^r}(t)]=d|G|^{\frac{d^n-1}{d-1}}\leq |G|^{d^n}$. Then by Theorem \ref{FPP}  we have,
	\[\left|\#\varphi^n\left(\bP^1(\F_q)\right)-\epsilon_n q\right|<\epsilon_n M|G|^{d^n}nq^{1/2}.\]
	By Lemma \ref{cosetasymptotic}, $\epsilon_n =\frac{2}{n\Phi''(1)}+O_d\left(\frac{\log n}{n^2}\right)$. By the proof of Lemma \ref{cosetasymptotic}, $\epsilon_n<\frac{C_2}{n}$ for some constant $C_2$ depending on $\Phi$. Let $C$ be the maximum over all such constants for transitive subgroups $G$ of $S_d$ and $\tau\in S_d$ such that  $\tau G$ is transitive, this constant depends only on $d$. Substituting into the last equation,
	\[\left|\#\varphi^n\left(\bP^1(\F_q)\right)-\epsilon_n q\right|< CM|G|^{d^n}q^{1/2}.\]

\end{proof} 

\begin{theorem}[Restatement of Theorem \ref{imbounds}] Suppose $K_n$ is tamely ramified over $\F_q(t)$ and  $\Gal\left(K_n/\F_q(t)\right) \cong [G]^n$, where $G=C_d, S_d,$ or $A_d$ for any $d\geq 2$. Then \[\#\varphi^n\left(\bP^1(\F_q)\right)=\epsilon_n q+O\left(|G|^{d^n}q^{1/2}\right),\] 
where $\epsilon_n=\FPP([G]^n)$. Moreover, $\epsilon_n=\frac{2}{nc_G}+O\left(\frac{\log n}{n^2}\right)$ if $n\geq 2$, where $c_{S_d}=c_{A_d} = 1$ and $c_{C_d}=d-1$.
\end{theorem}

\begin{proof}
	For $\varphi$, $K_n$ as in the theorem, $K_n/\F_q(t)$ must be a geometric extension by \cite[Proposition 3.6]{JKMT}. Hence by Theorem \ref{FPP}, \[\left|{\#\varphi^n\left(\bP^1(\F_q)\right)}-{\epsilon_n}q\right|<{\epsilon_n}M[K_n:\F_{q}(t)]ndq^{1/2},\]
	where $\epsilon_n=\FPP([G]^n)$. By Propositions \ref{Cbound}, \ref{Sbound}, and \ref{Abound}, $\epsilon_n<\frac{2}{n+1-\log n}$. So we can see $\epsilon_n n<3$. 
	Substituting we see \[\left|{\#\varphi^n\left(\bP^1(\F_q)\right)}-{\epsilon_n}q\right|<3M|G|^{d^n}q^{1/2}.\]
	Since we have shown for each group $\frac{2}{c_G(n+4+\log n)}<\epsilon_n<\frac{2}{c_G(n+1-\log n)}$ we can see $|\epsilon_n-\frac{2}{nc_G}|<\frac{2\log n +8}{c_G(n^2+4n+n\log n)}< \frac{4.05\log n}{n^2}$ for $n\geq 2$. 
	
\end{proof}

Now we describe rather mild conditions on the critical orbits of $\varphi$ that allow us to use the above results.
 Let $k$ be a field, $\varphi\in k(x)$, and $t$ transcendental over $k$. Let $K_n=k(\varphi^{-n}(t))$ and let ${crit}_\varphi$ denote the set of critical points of $\varphi$ in $\bar{k}$.

\begin{theorem}[\cite{JKMT}, Theorem 3.1] Let $k$ be a field, and suppose $K_1\cap \bar{k} = k$. Fix $N\in\mathbb{N}$ and suppose 
	$\varphi^n(a)\neq \varphi^m(b$) for $a,b\in{crit}_\varphi$ and $n,m\leq N$ unless $a= b$ and $n= m$ (i.e. there are no critical orbit relations).
	Then $\Gal(K_N/k(t))=[G]^N$ where $G=\Gal(K_1/k(t))$.	
\end{theorem}

\begin{theorem}[\cite{JO}, Theorem 3.1]
	Let $k$ be a field with $\ch k \neq 2$. Suppose $\Gal(K_1/k(t))\cong S_d$, there is some $a \in {crit}_\varphi$ with multiplicity one, and for all $b\in{crit}_\varphi$ and all $m\leq n\leq N$, $\varphi^n(a)\neq \varphi^m(b)$ unless $m=n$ and $b=a$. Then $\Gal(K_N/k(t))\cong [S_d]^N$.
\end{theorem}

Using these results and the results of this section we have the following.

\begin{corollary}\label{thm:conditions} Let $\varphi(x)\in \F_q(x)$. Suppose $K_{n}/\F_q(t)$ is tamely ramified. Then we have the following.
	\begin{enumerate}
		\item[(a)] If $\varphi$ satisfies the conditions in  \cite[Theorem 3.1]{JKMT} and $G=\Gal(K_1/\F_q(t))$, then for all $n\leq N$, \[\#\varphi^n(\bP^1(\F_q))=\epsilon_nq+O_d(|G|^{d^n}q^{1/2}),\]
		where  $\epsilon_n=\frac{2}{n\Phi_G''(1)}+O_d\left(\frac{\log n}{n^2}\right)$ for $n\geq 2$.
		
		\item[(b)] If  $\varphi$ satisfies the conditions in  \cite[Theorem 3.1]{JO}, then for all $n\leq N$, \[\#\varphi_{\fp}^n(\bP^1(\cO_K/\fp))=\epsilon_nq+O(d!^{d^n}q^{1/2}),\]
		where  $\epsilon_n=\frac{2}{n}+O\left(\frac{\log n}{n^2}\right)$ for $n\geq 2$.
	\end{enumerate}
	
\end{corollary}

\begin{proof}
	This follows immediately from Theorem~\ref{cor1}, Theorem \ref{imbounds}, \cite[Theorem 3.1]{JKMT}, and \cite[Theorem 3.1]{JO}.
\end{proof}

 We finish this section by using our results to bound the number of iterates that can occur before the critical orbits collide or cycle.

\begin{corollary}\label{prop:orbits}
	Let $\varphi(x)\in \F_q(x)$ have degree $d$ and suppose $\varphi(x)\in \F_q(x)$ is not a bijection. Then there is some $N$ depending on $d$ such that $\varphi^i(a)=\varphi^j(b)$ for some $i, j <\frac{Nq}{\log\log q}$ and critical points $a, b$ with $(a,i)\neq (b,j)$.
\end{corollary}

\begin{proof}
	If $\Gal(K_1^*/\bar{\F}_q(t))=G$ and $\varphi^i(a)\neq \varphi^j(b)$ for all $i, j\leq n$ unless $a=b$ and $i=j$ then by \cite[Theorem 3.1]{JKMT} we have $\Gal(K_n^*/\bar{\F}_q(t))=[G]^n$.  On the other hand, if $\Gal(K_n^*/\bar{\F}_q(t))=[G]^n$ for $n = \left\lfloor\frac{\log\log q}{\log d}\right\rfloor-2$ then Theorem \ref{cor1} implies, $\epsilon_nq=O_d\left(\frac{q}{\log\log q}\right)$. Since $n\leq\frac{\log\log  q}{\log d}-2$, 
	we have $|G|^{d^n}\leq d!^{d^n}\leq d!^{\frac{\log  q}{ d^2}}=q^{\frac{\log d!}{d^2}}<q^{1/4}$. Hence using Theorem~\ref{cor1} again, we conclude there is some constant depending on $d$ such that
	\[\#\varphi^n(\bP^1(\F_q))=\frac{Cq}{\log\log q}.\] 
	
	If $k=\#\varphi^n(\bP^1(\F_q))+1$, then for any critical point $a$, the points $\varphi^n(a), \varphi^{n+1}(a), \dots, \varphi^{n+k}(a)$ cannot all be distinct. Thus, $\varphi^i(a)=\varphi^j(a)$ for some $i<j<n+k$, where $n+k=\left\lfloor\frac{\log\log q}{\log d}\right\rfloor-1+\frac{Cq}{\log\log q}=O_d\left(\frac{q}{\log\log q}\right)$.
\end{proof}


\section{Examples}\label{Examples}

In this section we apply our results to two families of polynomials.  Theorem~\ref{example} follows from the next theorem.


\begin{theorem}Let $d>1$ and consider $\varphi(x)=ax^d+c\in \F_q(x)$, where $a,c\neq 0$. Suppose  $\varphi^n(0)\neq \varphi^m(0)$ for all $i<j\leq n$.
	 If $q\equiv 1\mod d$, then  
	\[\#\varphi^n\left(\bP^1(\F_q)\right)=\epsilon_nq+O\left(d^{d^n}\sqrt{q}\right),\text{ where } \epsilon_n=\frac{2}{(d-1)n}+O\left(\frac{\log n}{n^2}\right) \text{ for $n\geq 2$}.\] 
	 If $q\equiv 1\mod \ell$ for any $\ell|d$ with $\ell>1$, then 	\[\#\varphi^n\left(\bP^1(\F_q)\right)=\epsilon_nq+O_d\left(d^{d^n}\sqrt{q}\right), \text{ where } \epsilon_n=\frac{2}{n\Phi_{\sigma G}''(1)}+O_d\left(\frac{\log n}{n^2}\right)\text{ for $n\geq 2$,}\] 
	 where $G=\Gal(K_1/\F_q(\zeta_d,t))\cong C_d$ and $\sigma\in G$ with $\sigma|_{\F_q(\zeta_d)}(x)=x^q$. 
	Otherwise, $\varphi_q:\F_q\rightarrow \F_q$ is a bijection, so $\varphi^n(\F_q)=\F_q$ for all $n$.
	
	Further, if $a,c \in\bZ^{+}$ we can consider $\varphi(x)\in \bZ(x)$ and 
	$\varphi_q(x)=ax^d+c \in \F_q(x)$, the reduction of $\varphi_q$ of $\varphi$ modulo $q$.
	Then the first equation holds for all $\F_q$ with  $\ch \F_q>(a+c)^{\frac{d^n-1}{d-1}}$ and $q\equiv 1 \mod d$. 
\end{theorem}

\begin{proof}	
	If $q\equiv 1\mod d$, $\F_q$ contains a $d$-th root of unity so $K_1\cap \bar{\F}_q =\F_q$. Then the hypotheses of \cite[Theorem 3.1]{JKMT} hold and Corollary \ref{thm:conditions} implies
	$\#\varphi^n(\bP^1(\F_q))=\epsilon_nq+O(d^{d^n}q^{1/2}),\text{ where } \epsilon_n= \frac{2}{(d-1)n}+O\left(\frac{\log n}{n^2}\right).$
	
	If $\F_q$ contains any $\ell$-th root of unity for $\ell|d$ with $\ell>1$ then $K_1\cap\bar{\F}_q=\F_q(\zeta_d)$ let $\sigma\in \Gal(K_n/\F_q(t))$ be an element with the property $\sigma|_{\F_q(\zeta_d)}=\Frob_q$. Then by Theorem~\ref{cor1},
	$\#\varphi^n(\bP^1(\F_q))=\epsilon_nq+O_d(d^{d^n}q^{1/2})$  where $\epsilon_n=\frac{2}{n\Phi_{\sigma G}''(1)}+O_d\left(\frac{\log n}{n^2}\right)$
	and $\sigma G=(\sigma|_{K_1})\Gal(K_1/\F_q(\zeta_d,t))$.
	
	If $\F_q$ does not contain any $\ell$-th roots of unity for $\ell >1$ dividing $d$, then $\varphi:\F_q\rightarrow\F_q$ is a bijection so $\varphi^n(\F_q)=\F_q$ for all $n$.
 
	 Note, if $a,c \in\bZ^{+}$ we can consider $\varphi(x)\in \bZ(x)$ and 
	 $\varphi_q(x)=ax^d+c \in \F_q(x)$, the reduction of $\varphi_q$ of $\varphi$ modulo $q$.
	 A simple induction shows $0<\varphi^{n-1}(0)<\varphi^n(0)\leq (a+c)^{\frac{d^n-1}{d-1}}$ for all $n$.  So we have 
	 \[\#\varphi_q^n(\bP^1(\F_q))=\epsilon_nq+O(d^{d^n}q^{1/2}),\] for all $\F_q$ with  $\ch \F_q>(a+c)^{\frac{d^n-1}{d-1}}$ and $q\equiv 1 \mod d$. 
\end{proof}


	
	\begin{theorem}\label{example2} Let $d\geq 3$ and consider $\varphi(x)=(d-1)x^d+(da)x^{d-1}\in \F_q(x)$,  where $ad(d-1)\neq 0$. Suppose  $\varphi^i(-a)\neq \varphi^j(-a)$ for all $i<j\leq n$. Then
		\[\#\varphi^n\left(\bP^1(\F_q)\right)=\epsilon_nq+O\left(d!^{d^n}\sqrt{q}\right),\text{ where } \epsilon_n=\frac{2}{n}+O\left(\frac{\log n}{n^2}\right)\text{ for $n\geq 2$}.\] 
	Further, let $2\leq a \in\bZ$, $\varphi(x)=(d-1)x^d+(da)x^{d-1}\in \bZ(x)$, and $\varphi_q(x)=(d-1)x^d+(da)x^{d-1} \in \F_q(x)$, the reduction of $\varphi$ modulo $q$.
	If $\ch \F_q> (2d)^{\frac{d^{n-1}-1}{d-1}}a^{d^n}$ and  $\ch(\F_q)\nmid d$ then  
	the above equation holds.
		
	\end{theorem}

	We first prove a proposition using the following two lemmas. These are  fairly standard results from algebraic number theory stated here without proof. Results similar to the next lemma  can also be found in \cite{GTZ} and  \cite{vanderWaerden}.
	
		\begin{lemma}[\cite{JO}, Lemma 2.4] \label{inertia group}
		Let $M/K$ be a finite Galois extension with Galois group $G$. Let $H$ be a subgroup of $G$ and $L = M^H$ be the corresponding intermediate field. Let $\mathfrak{q}$ be a prime of $M$ and $\mathfrak{p} := \mathfrak{q} \cap K$. Let $X$ be the transitive $G$-set $G/H$. Then there is a bijection between the set of orbits of $X$ under the action of $D(\fq|\fp)$, the decomposition group of $\mathfrak{q}$ over $\mathfrak{p}$, and the set of extensions $\mathfrak{P}$ of $\mathfrak{p}$ to $L$ with the property: If $\mathfrak{P}$ corresponds to $Y$ then the length of $Y$ is $e(\mathfrak{P}|\mathfrak{p})f(\mathfrak{P}|\mathfrak{p})$ and $Y$ is the disjoint union of $f(\mathfrak{P}|\mathfrak{p})$ orbits of length $e(\mathfrak{P}|\mathfrak{p})$ under the action of $I(\mathfrak{q}|\mathfrak{p})$, the inertia group of $\mathfrak{q}$ over $\mathfrak{p}$.  
	\end{lemma}

	We will consider the case $K=k(t)$, $L = K(\theta)$ where $\theta$ is a root of $\varphi(x)-t$. Then $M=K(\varphi^{-1}(t))$. Further, since the set $G/H$ corresponds to the set of $K$ homomorphisms of $K(\theta)$ into $K(\varphi^{-1}(t))$, the elements of $G/H$ correspond to roots of $\varphi$ in $K(\varphi^{-1}(t))$. In this case, Lemma \ref{inertia group} implies there is a one-to-one correspondence between the set of orbits of the roots of $\varphi$ under the action of the decomposition group $D(\fq|\fp)$ and the set of extensions of $\fp$ to $K(\theta)$ with the property from Lemma \ref{inertia group}. 
	
	\begin{lemma}[Kummer's Theorem, see \cite{Jan}, Theorem 7.4]\label{splitting} Let $R$ be an integral domain with fraction field $K$, $L$ be a finite extension of $K$, and $R'$ be the integral closure of $R$ in $L$. Let $\fp$ be a nonzero prime ideal in $R$ and $\theta$ an element of $L$ such that the integral closure of $R_\fp$ in $L$ is $R_\fp[\theta]$. Suppose $f(x)$ is the minimal polynomial of $\theta$ over $K$, $\bar{f}(x)$ is the reduction of $f(x)$ modulo $\fp$, and $\bar{f}(x)$ factors into distinct irreducible polynomials 
		\[\bar{f}(x) = g_1(x)^{e_1}\dots g_t(x)^{e_t},\]
		then $\fp R'$ factors as $\fp R'=\fP_1^{e_1}\dots \fP_t^{e_t}$ where $f(\fP_i|\fp)=\deg(g_i)$.
	\end{lemma}

	\begin{proposition}\label{prop:example2} Let $k$ be an algebraically closed field. Let $\varphi(x)=(d-1)x^d+(da)x^{d-1}\in k(x)$. Suppose $a,d,(d-1)\neq 0$. Then  $\Gal(k(\varphi^{-1}(t))/k(t))\cong S_d$.
	\end{proposition}

	\begin{proof} Using the Riemann-Hurwitz formula it is easy to see that $k(t)$ has no extensions of degree $d>1$ which are unramified at all finite primes and tamely ramified at the prime at infinity, see \cite[Lemma 2.11]{JO}. Consider the group \[I=<I(\fq|\fp)>_{\mathfrak{p} \in \mathbb{P}_{k(t)}\setminus \{\mathfrak{p}_\infty\}}\]
	which is a subgroup of $\Gal(k(\varphi^{-1}(t))/k(t))$. Since the fixed field of $I$ is an unramified extension of $k(t)$, we  have $k(\varphi^{-1}(t))^I=k(t)$ and hence $I= \Gal(k(\varphi^{-1}(t))/k(t))$. We study the ramified primes. 
		
	The critical points of $\varphi$ are $0$ and $-a$, which have multiplicities $d-1$ and $1$ respectively. From the polynomial discriminant formula, we see that the primes of $k(t)$ ramifying in $k(\varphi^{-1}(t))$ are $\varphi(0)-t$ and $\varphi(-a)-t$ see \cite{CH}. 
		
	First, consider $\fp =(\varphi(0)-t)=(t)$. Let $\theta$ be any root of $\varphi(x)-t$ then since $k$ is algebraically closed, $\theta$ is integral over $k[t]$ and the integral closure of $k[t]$ in $k(\theta)$ is $k[\theta]$. 
	Since $\varphi(x)-t \equiv \varphi(x) \equiv x^{d-1}\left((d-1)x+da\right)\mod \fp$, we have $\fp k[\theta] = \fP_1^{d-1}\fP_2$ by Lemma \ref{splitting}. Then Lemma \ref{inertia group} implies that for any $\fq$ in $k(\varphi^{-n}(t))$ lying over $\fp$, $I(\fq|\fp)$ acts transitively on $d-1$ roots of $\varphi(x)-t$ and fixes the remaining root. 
		
	Now consider $\fp'=(\varphi(-a)-t) =(\pm a^d-t)$. Since $-a$ is a critical point of multiplicity one, for any $\fq$ lying above $\fp'$, $I(\fq|\fp')$ is generated by a single transposition \cite[Corollary 2.8]{JO}.  
		
	We see from the above arguments that $I=\Gal(k(\varphi^{-1}(t))/k(t))$ is a subgroup of $S_d$ containing at least one transposition and at least one subgroup acting transitively on $d-1$ elements.
	Also, since $I$ is the Galois group and $\varphi(x)-t$ is irreducible, $I$ must be transitive. 
		
	Suppose $I'\subset I$ acts transitively on $\{\theta_1,
	\dots, \theta_{d-1}\}$ and fixes $\theta_d$. Let $(\theta_i, \theta_j)$ be any transposition in $I$. Since $I$ is transitive we can find $\sigma \in I$ such that $\sigma(\theta_i)=\theta_d$. Then $\sigma(\theta_i, \theta_j)\sigma^{-1} = (\sigma(\theta_j), \theta_d)\in I$. Denote $\sigma(\theta_j)=\theta_k$. 
	We claim that $<I',  (\theta_k, \theta_d)>\cong S_d$ and hence $I\cong S_d$. For any $\theta_\ell\in \{\theta_1,
	\dots, \theta_{d-1}\}$ we can find  $\tau \in I'$  sending $\theta_k$ to $\theta_\ell$ then $\tau(\theta_k, \theta_d)\tau^{-1} = (\theta_\ell, \theta_d)$. Thus, any transposition of the form $(\theta_\ell, \theta_d)$ is in  $<I',  (\theta_k, \theta_d)>$, it is a standard exercise in group theory to show that these transpositions generate $S_d$. 
	\end{proof}

	\begin{proof}[Proof of Theorem \ref{example2}]
		Suppose $\varphi_q^n(-a)\neq \varphi_q^m(-a)$ for any $m\leq n$. By Proposition \ref{prop:example2}, 
		$\Gal(\bar{\F}_q(\varphi^{-1}(t))/\bar{\F}_q(t))\cong S_d$. Since $S_d\cong\Gal(\bar{\F}_q(\varphi^{-1}(t))/\bar{\F}_q(t))\subseteq \Gal({\F}_q(\varphi^{-1}(t))/{\F}_q(t))$, we have  $\Gal({\F}_q(\varphi^{-1}(t))/{\F}_q(t))\cong S_d$. Hence by Corollary \ref{thm:conditions},\[\#\varphi_{q}^n(\bP^1(\F_q))=\epsilon_nq+O(d!^{d^n}q^{1/2}),\]
		where  $\epsilon_n=\frac{2}{n}+O\left(\frac{\log n}{n^2}\right)$. 
		
		Now let $a\in\bZ$ with $a\geq 2$ and consider $\varphi(x)\in \bZ(x)$.
		We claim the orbit of $-a$ is infinite. To see this, note if $d$ is odd then $\varphi(-a) = a^d$. Then since $\varphi(x)$ is a strictly increasing function on the interval $(0,\infty)$, we have $\varphi^n(-a)>\varphi^{n-1}(-a)$ for all $n$. If $d$ is even, then $\varphi(-a) = -a^d$ and $\varphi^2(-a)=(d-1)a^{d^2}-da^{d^2-d+1}=a^{d^2-d+1}((d-1)a^{d-1}-d)>0$. Again, since $\varphi(x)$ is increasing on the interval $(0,\infty)$, $\varphi^n(-a)>\varphi^{n-1}(-a)$ for all $n>2$. We can see by induction that $\varphi^n(-a)<(2d)^{\frac{d^{n-1}-1}{d-1}}a^{d^n}$ for all $n>1$.
		Fix $q$ with  $\ch \F_q> (2d)^{\frac{d^{n-1}-1}{d-1}}a^{d^n}$. Then $\varphi_q^n(-a)\neq \varphi_q^m(-a)$ for any $m\leq n$, so the result holds.
		
	\end{proof}
	


\section{Explicit Bounds on Proportions of Periodic Points}\label{BoundsPer}

In this section we prove Theorem \ref{theorem.periodicbd}.  Fix a polynomial $\varphi(x)\in\cO_K[x]$ where $\cO_K$ is the ring of integers of a number field $K$. Suppose that the critical points of $\varphi(x)$ belong to $\cO_K$. We label the set of critical points ${crit}_\varphi$. 
Suppose further that 
\begin{equation}\label{criterion2}
\varphi^n(a)\neq \varphi^m(b) \text{  for $a,b\in{crit}_\varphi$ and $n,m\in\mathbb{N}$ unless $a= b$ and $n= m$}.
\end{equation}  

We start with a definition and lemma that we will need to prove part (b) of the theorem.

\begin{definition} Let $K$ be a number field and $M_K$ the set of places of $K$. For any $\alpha\in K$ define the \textbf{height} of $\alpha$ to be $H(\alpha) =\left(\prod_{v\in M_K} \max\{|\alpha|_v,1\}^{n_v}\right)^\frac{1}{[K:\mathbb{Q}]}$ where $n_v=[K_v:\mathbb{Q}_v]$. Here $K_v$ and $\bQ_v$ denote the $v$-adic completions of $K$ and $\bQ$ respectively. Define the \textbf{logarithmic height} of $\alpha$ to be $h(\alpha)=\log H(\alpha)$ (see \cite{BG} or \cite{Sil}). 
\end{definition}

\begin{lemma}\label{lemma.distinctorbits}
	Let $\varphi(x) \in \cO_K[x]$ as above and let $d$ be the degree of $\varphi$. Then there exists a constant $B$ such that $H(\varphi^r(c))<B^{d^n}$ for all $r\leq n$ and all $c\in{crit}_\varphi$. Further, if $\fp$ is a prime of $\cO_K$ with $N(\fp)>  (2B^{2d^n})^{[K:\mathbb{Q}]}$, all of the points in the critical orbits up to the $n$-th iterate will remain distinct in $\cO_K/\fp$, where $N(\fp)=|\cO_K/\fp|$.
\end{lemma}

\begin{proof}
	By \cite[Theorem 3.11]{Sil}, for any polynomial $\varphi(x)\in K[x]$ of degree $d$, there are explicitly computable constants $C_1, C_2$ depending only on $\varphi$ such that \[C_1H(\alpha)^d\leq H(\varphi(\alpha))\leq C_2H(\alpha)^d \text{ for all } \alpha \in \bP^1(\bar{K}).\]
	We can rewrite this using the logarithmic height as
	\[\log C_1+dh(\alpha)\leq h(\varphi(\alpha))\leq \log C_2+dh(\alpha),\]
	and taking $C=\max\{|\log{C_1}|, |\log{C_2}|\}$, we see
	\[|h(\varphi(\alpha))-dh(\alpha)|\leq C \text{ for all } \alpha \in \bP^1(\bar{K}).\]
	
	Then following the arguments in the proof of \cite[Theorem 3.20]{Sil} we see 
	
	\begin{align*}
	\left|\frac{1}{d^r}h(\varphi^r(\alpha)) - h(\alpha)\right| & = \left|\sum_{i=1}^r \frac{1}{d^i}(h(\varphi^i(\alpha))-dh(\varphi^{i-1}(\alpha)))\right|\\
	&\leq \sum_{i=1}^r \left|\frac{1}{d^i}(h(\varphi^i(\alpha))-dh(\varphi^{i-1}(\alpha)))\right|\\
	&\leq \sum_{i=1}^r \frac{C}{d^i} \leq \sum_{i=1}^\infty \frac{C}{d^i}= \frac{C}{d-1}.\\
	\end{align*}
	Thus, $\frac{1}{d^r}h(\varphi^r(\alpha))\leq h(\alpha) +\frac{C}{d-1}$ for all $\alpha\in K$ and $r\leq n$. 
	
	Also since, ${crit}_\varphi$ is finite we can find a constant $D$ such that $h(c)<D$ for all $c\in {crit}_\varphi$.
	If $c\in {crit}_\varphi$, then we have $h(\varphi^r(c))\leq d^r\left(h(c)+\frac{C}{d-1}\right)\leq d^r\left(D+\frac{C}{d-1}\right)$. Let $B=e^{D+\frac{C}{d-1}}$, then $H(\varphi^r(c))\leq B^{d^n}$ for all $r\leq n$ and all $c\in{crit}_\varphi$. 
	
	Now let $\fp$ be a prime of $\cO_K$ lying over $p\in \mathbb{Z}$. Let $q=p^{\deg \fp}=|\cO_K/\fp|$. Note if $\alpha\equiv \beta \mod \fp$ for $\alpha, \beta\in \cO_K$, then $q=N(\fp)$ divides $N_{K/\bQ}(\alpha-\beta)=\prod_{\sigma} \sigma(\alpha-\beta)$, where the product is taken over the distinct embeddings of $K$ into $\bar{\bQ}$. Hence, $q\leq |N_{K/\mathbb{Q}}(\alpha-\beta)|_\infty =H(\prod_{\sigma} \sigma(\alpha-\beta))=H(\alpha-\beta)^{[K:\mathbb{Q}]}\leq (2H(\alpha)H(\beta))^{[K:\mathbb{Q}]}$ where the last inequality follows from \cite[Proposition 1.5.15]{BG}. 
	Thus, if $\varphi^r(a)\equiv \varphi^m(b)\mod \fp$ for $m,r\leq n$ and $a, b \in {crit}_\varphi$, then $N(\fp)\leq (2B^{2d^n})^{[K:\mathbb{Q}]}$.
	
	Taking $q=N(\fp)>  (2B^{2d^n})^{[K:\mathbb{Q}]}$ ensures all of the points in the critical orbits will remain distinct in $\cO_K/\fp$. It suffices to choose $n< \frac{\log(\log q-[K:\bQ]\log 2)-\log(2[K:\bQ]\log B)}{\log d}$.
\end{proof}

\begin{proof}[Proof of Theorem \ref{theorem.periodicbd}]
	
If each coset of $\Gal\left(K(\varphi^{-1}(t))/K(t)\right)/\Gal\left(\bar{K}(\varphi^{-1}(t))/\bar{K}(t)\right)$ contains at least one fixed point free element, then the same is true for the reduction to $\cO_k/\fp=\F_q$ for any $\fp$. Also by \cite[Proposition 4.1]{JKMT}, $\Gal\left(\bar{\F}_q(\varphi^{-1}(t))/\bar{\F}_q(t)\right)\cong[G]^n$ for all but finitely many primes $q$. For these primes, by Corollary \ref{thm:conditions}, 
\[\#\varphi^n(\bP^1(\F_q))=\epsilon_nq+O(|G|^{d^n}q^{1/2}),\]
where  $\epsilon_n=\frac{2}{n\Phi_G''(1)}+O_d\left(\frac{\log n}{n^2}\right)$. 
 Choose $n=\left\lfloor \frac{\log\log q}{\log d}-2\right\rfloor$, then $\frac{1}{n}=O_d\left(\frac{1}{\log\log q}\right)$ and $|G|^{d^n}<q^{1/4}$. So we can see	\[\frac{\#\Per(\varphi_\fp)}{q+1}\leq \frac{\#\varphi^n(\bP^1(\F_q))}{q+1}=O_d \left(\frac{1}{\log\log q}\right).\]
	

 Now assume the hypotheses of part (b) of the theorem. Let $A=\max\{d^2, \log B^{2[K:\bQ]}\}$. Fix a prime $\fp$ in $\cO_K$ such that $\left|\cO_k/\fp\right|=q\geq 2^{[K:\bQ]}e^{A}$. Take $n=\left\lfloor \frac{\log(\log  q-[K:\bQ]\log 2)-\log A}{\log d} \right\rfloor\geq 0$. Then $q\geq (2B^{2d^n})^{[K:\mathbb{Q}]}$ and by Lemma \ref{lemma.distinctorbits}, the critical orbits remain distinct in $\cO_K/\fp=\F_q$ up to the $n$-th iterate.  Applying \cite[Theorem 3.1]{JKMT}, we see $\Gal(K_n/{\F}_q(t))=[G]^n$ where $G=\Gal(K_1/{\F}_q(t))$.

By Theorem \ref{FPP}, 
\[\#\varphi^n(\bP^1(\F_q))<\epsilon_nq+6[K_n:\F_q(t)]ndq^{1/2},\]
where $\epsilon_n=\FPP([G]^n)$.
 Suppose $G=C_d$ or $S_d$, or $G=A_d$ and $d\neq 4$, so we have $\epsilon_n<\frac{2}{n+1}$.
Since    $n<\frac{\log\log  q -2\log d}{\log d}$, 
we have \[[K_n:\F_q(t)]d=|G|^{d^n}< d!^{d^n}< d!^{\frac{\log  q}{ d^2}}=q^{\frac{\log d!}{d^2}}<q^{1/4}.\] 
Also since  $\frac{1}{n+1}<\frac{\log d}{\log(\log  q-[K:\bQ]\log 2)-\log A}$, we have 
\[\frac{\#\Per(\varphi_q)}{q+1}\leq \frac{\#\varphi^n(\bP^1(\F_q))}{q+1}<\frac{2\log d}{\log(\log  q-[K:\bQ]\log 2)-\log A}+12q^{-1/4}.\]

\end{proof}

\end{document}